\documentclass[11pt]{article}
\usepackage[a4paper, margin=1in]{geometry}
\usepackage{amsfonts,amsmath,amssymb,amsthm,graphicx,fixmath,latexsym,color}
\usepackage{xcolor}
\usepackage{hyperref,epsfig}
\usepackage{algorithmic}
\usepackage{algorithm}
\usepackage{xspace}

\providecommand{\remove}[1]{}

\newcommand{\R}{\mathcal{R}}

\newcommand{\A}{\mathcal{A}}

\newcommand{\F}{\mathcal{F}}

\newcommand{\h}{\mathcal{H}}

\newtheorem{theorem}{Theorem}[section]

\newtheorem{proposition}[theorem]{Proposition}

\newtheorem{definition}[theorem]{Definition}
\newtheorem{remark}[theorem]{Remark}

\newtheorem*{theorem*}{Theorem}
\newtheorem*{lemma*}{Lemma}
\newtheorem*{proposition*}{Proposition}
\newtheorem{observation}[theorem]{Observation}

\begin{document}
\title{On the Union Complexity of Families of Axis-Parallel Rectangles with a Low Packing Number}

\author{Chaya Keller\thanks{Department of Mathematics, Ben-Gurion University of the NEGEV, Be'er-Sheva Israel. \texttt{kellerc@math.bgu.ac.il}. Research partially supported by Grant 635/16 from the Israel Science Foundation, the Shulamit Aloni Post-Doctoral Fellowship of the Israeli Ministry of Science and Technology, and by the Kreitman Foundation Post-Doctoral Fellowship.}
\and
Shakhar Smorodinsky\thanks{Department of Mathematics, Ben-Gurion University of the NEGEV, Be'er-Sheva Israel. \texttt{shakhar@math.bgu.ac.il}. Research partially supported by Grant 635/16 from the Israel Science Foundation.}
}

\date{}
\maketitle

\begin{abstract}
Let $\R$ be a family of $n$ axis-parallel rectangles with packing number $p-1$, meaning that among any $p$ of the rectangles, there are two with a non-empty intersection. We show that the union complexity of $\R$ is at most $O(n+p^2)$, and that the $(\leq k)$-level complexity of $\R$ is at most $O(kn+k^2p^2)$. Both upper bounds are tight.
\end{abstract}

\section{Introduction}

For a finite family $\mathcal{C}=\{C_1,C_2,\ldots,C_n\}$ of geometric objects in the plane, the \emph{union complexity} of $\mathcal{U(C)} = \cup_{i=1}^n C_i$ (or, in short, the union complexity of $\mathcal{C}$) is the number of \emph{vertices} on the boundary $\partial(\mathcal{U(C)})$, where a vertex is an intersection point of the boundaries of two objects $C_i,C_j \in \mathcal{C}$.\footnote{Formally, the definition of the union complexity is slightly more complex: it is the total number of faces of all dimensions of the arrangement of the boundaries of the
objects, which lie on the boundary of the union (see~\cite{APS08}). We use our simpler definition as in our context, both definitions are clearly equivalent up to a constant factor.} More generally, for any $k \geq 0$, the \emph{$(\leq k)$-level complexity} of $\mathcal{C}$ is the number of vertices that are contained in the interior of at most $k$ elements of $\mathcal{C}$.

Bounding the union complexity of families of geometric objects is useful for analyzing the running time of various algorithms, and has applications to linear programming, robotics, molecular modeling, and many other fields. In particular, Clarkson and Varadarajan~\cite{CV07} showed that if the union complexity of a family $\R$ of $r$ ranges with $VC$ dimension $\delta$ is sufficiently close to $O(r)$, then $\R$ has an $\epsilon$-net of size smaller than $O(\frac{\delta}{\epsilon}\log \frac{\delta}{\epsilon})$. Smorodinsky~\cite{S07} showed that bounds on the union complexity and on the level-1 complexity of families of geometric objects in the plane can be used in computing the proper chromatic number and the conflict-free chromatic number of the corresponding hypergraph. For more on union complexity, see the survey~\cite{APS08}.

For several families of geometric objects, it was shown that the union complexity is asymptotically lower than the trivial $O(n^2)$ bound. In particular,
Kedem et al.~\cite{KLPS86} showed that the union complexity of any family of $n$ pseudo-discs in the plane is at most $6n-12$, and Alt et al.~\cite{AFKMNSU92} and Efrat et al.~\cite{ERS93} proved a similar bound for any family of fat wedges. An almost linear bound for families of $\gamma$-fat triangles was obtained by Ezra et al.~\cite{EAS11}.

For a general family of axis-parallel rectangles in the plane, the union complexity can be quadratic -- e.g., if the family is an $\frac{n}{2}$-by-$\frac{n}{2}$ grid of long and thin rectangles. However, one may note that such a family contains as many as $n/2$ pairwise disjoint sets. Hence, it is natural to ask whether any family of axis-parallel rectangles with a quadratic union complexity must contain a linear-sized sub-family whose elements are pairwise disjoint.

%Hence, it is interesting to see which additional conditions on the family can assure a sub-quadratic union complexity.

In this note we answer this question on the affirmative. We show that the union complexity of any family $\R$ of axis-parallel rectangles is sub-quadratic if the \emph{packing number} of the family is sub-linear. Recall that the packing number of $\R$, denoted $\nu(\R)$, is $p-1$ if among any $p$ elements of $\R$, two have a non-empty intersection. Our main result is the following:
\begin{theorem}\label{Thm:Main}
Let $\R$ be a family of axis-parallel rectangles with packing number $\nu(\R)$. Then for any $k \geq 0$, the $(\leq k)$-level complexity of $\R$ is $O(kn+ k^2\nu(\R)^2)$. In particular, the union complexity of $\R$ is $O(n+ \nu(\R)^2)$.
\end{theorem}
\noindent Both results are tight, as we show by an explicit example.\footnote{We note that our upper bound on the union complexity is not hereditary, in the sense that there may exist a sub-family of $\R$ (of size $\Theta(p)$) whose union complexity is quadratic in its number of elements. Another non-hereditary bound on the union complexity, for specific families of discs in the plane, was obtained recently by Aronov et al.~\cite{ADPS13}.}

\section{Proof of Theorem~\ref{Thm:Main}}
\label{sec:proof}

The proof of Theorem~\ref{Thm:Main} consists of several steps, and for convenience we divide them into separate subsections. We start with a few definitions and notations.

\subsection{Definitions and Notations}
\label{sec:def}

Throughout this note, $\R$ denotes a family of axis-parallel rectangles in the plane, and we assume that $\R$ is in \emph{general position}, meaning that no two rectangles have more than 4 common points (i.e., no two rectangles share a segment of the boundary; this implies that no three boundaries intersect at the same point). Put $\nu(\R)=p-1$, so any $p$ rectangles in $\R$ contain two with a non-empty intersection.

For any $x \in \mathbb{R}^2$, the \emph{depth} of $x$, denoted $\mathrm{depth}(x)$, is the number of rectangles in $\R$ that contain $x$ as an interior point. Let $Y_0$ be the set of \emph{vertices} (i.e., intersections of pairs of boundaries) of depth $0$, and for $k \geq 0$, let $Y_{\leq k}$ be the set of vertices of depth at most $k$. Of course, $|Y_0|$ is the union complexity of $\R$ and $|Y_{\leq k}|$ is the $(\leq k)$-level complexity of $\R$.

Intersection points of boundaries of two axis-parallel rectangles can be partitioned into four types, depicted in Figure~\ref{fig:types-of-intersection}. The type described in Figure~\ref{fig:types-of-intersection}(a) (in which the intersection point is the rightmost-upmost point of the intersection of the rectangles) will be called \emph{type L} intersection. We denote by $X_0$ the set of all points of type $L$ in $Y_0$, and by $X_{\leq k}$ the set of all points of type $L$ in $Y_{\leq k}$.

For any intersection point $x$ of type $L$, we denote by $A_x$ the rectangle to whose \emph{upper} boundary $x$ belongs, and by $B_x$ the rectangle to whose \emph{right} boundary $x$ belongs.

\begin{figure}[tb]
\begin{center}
\scalebox{0.6}{
\includegraphics{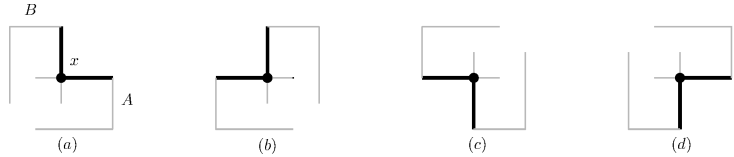}
} \caption{Types of intersection of pairs of axis-parallel rectangles in general position.}
\label{fig:types-of-intersection}
\end{center}
\end{figure}

%Consider the family $\R'$ of all projections of elements of $\R$ to the $x$ axis. $\R'$ is a family of segments on the line that satisfies the $(p,2)$ property (meaning that among every $p$ segments, there are two with a non-empty intersection), and hence, by the classical Hadwiger-Debrunner theorem~\cite{HD57}, there exists a set of $p-1$ points that \emph{pierces} $\R'$ (i.e., intersects any segment in $\R'$). Therefore, there exists a set of $p-1$ vertical lines that pierces all elements of $\R$. We pick such a set of lines $H=h_1,h_2,\ldots,h_{p-1}$, and use it throughout the note.

\subsection{Partition of $\R$ into floors}
\label{sec:sub:floors}

Let $R_1 \in \R$ be the rectangle whose upper boundary is the lowest (i.e., has the smallest $y$ coordinate) among the rectangles in $\R$. If there are several such rectangles, we choose one of them arbitrarily. Denote by $\ell_1$ the horizontal line that contains the upper boundary of $R_1$.

Define inductively a sequence $\{\ell_i\}_{2 \leq i \leq p-1}$ as follows. Let $R_i$ be the rectangle whose upper boundary is the lowest between all elements of $\R$ whose lower boundary is above $\ell_{i-1}$. (Again, if there are several such rectangles, we pick one of them arbitrarily.) Denote by $\ell_i$ the horizontal line that contains the upper boundary of $R_i$. Note that by the construction, the rectangles $\{R_i\}_{1 \leq i \leq p-1}$ are pairwise disjoint. As $\nu(\R)=p-1$, this implies that $\R$ does not contain any rectangle whose lower boundary is above $\ell_{p-1}$. Let $\ell_p$ be a horizontal line that lies above the upper boundaries of all the rectangles in $\R$ (such a line clearly exists as $\R$ is finite and all its rectangles are compact).

We now define the partition of $\R$ into floors: we say that $R \in \R$ belongs to floor $i$, $1 \leq i \leq p-1$, if the upper boundary of $R$ is above or contained in $\ell_i$ and \emph{lower} than $\ell_{i+1}$. We denote the set of all rectangles in floor $i$ ($1 \leq i \leq p-1$) by $\F_i$.
%and let $n_i=|\F_i|$.
It is clear from the construction that $\{\F_i\}_{1 \leq i \leq p-1}$ is a partition of $\R$ into $p-1$ pairwise disjoint families. In addition, we need the following observation:
\begin{observation}\label{Obs:Floors}
For any $1 \leq i \leq p-1$, if $R \in \F_i$ then $R \cap \ell_i \neq \emptyset$. Furthermore, $i$ is the largest index such that $R$ intersects $\ell_i$.
\end{observation}

\begin{proof}
Let $R \in \F_i$. If the lower boundary of $R$ is above $\ell_i$ then by the definition of $\ell_{i+1}$, the upper boundary of $R$ cannot lie strictly below $\ell_{i+1}$, a contradiction. Hence, the lower boundary of $R$ is either below $\ell_i$ or on $\ell_i$. As the upper boundary of $R$ is either on $\ell_i$ or above $\ell_i$ and also lower than $\ell_{i+1}$, the assertion follows.
\end{proof}

Observation~\ref{Obs:Floors} implies that $\R$ is \emph{pierced} by the set of lines $\mathcal{L} = \{\ell_1,\ldots,\ell_{p-1}\}$, meaning that each $R \in \R$ has a non-empty intersection with (at least) one of the lines. A similar argument shows that there exists a set $\h = \{h_1,h_2,\ldots,h_{p-1}\}$ of vertical lines (arranged in increasing order of the $x$ coordinate) that pierces $\R$. This set will be used, along with $\mathcal{L}$, in the sequel.

\subsection{Classification of the intersection points of type $L$}
\label{sec:sub:contributed}

In what follows, we obtain an upper bound on $|X_{\leq k}|$, i.e., the number of intersection points of type $L$ and depth $\leq k$. (By symmetry, this will imply an upper bound on the $(\leq k)$-level complexity of $\R$.) As a preparation, we classify the intersection points of type $L$. %Each such intersection point $x$ is formed by a rectangle $A \in \R$ such that $x$ lies on the upper boundary of $A$, and a rectangle $B \in \R$ such that $x$ lies on the right boundary of $B$. Our division is performed according to $A$, and to the rightmost vertical line $h_i \in \h$ that intersects $B$.

\begin{definition}
%Let $x \in X_{\leq k}$ be obtained as an upmost-rightmost point in $A \cap B$ for some $A,B \in \R$, and assume that $x$ lies on the upper boundary of $A$ and on the right boundary of $B$ (see Figure~\ref{fig:types-of-intersection}(a)).
Let $x \in X_{\leq k}$. Denote by $h_x$ the rightmost amongst the vertical lines in the set $\{h \in \h: h \cap B_x \neq \emptyset\}$. We say that $x$ is \emph{$(A_x,h_x)$-contributed}.

For $A \in \R$, we say that $x$ is $A$-contributed if there exists $h \in \h$ such that $x$ is $(A,h)$-contributed. Conversely, for $h \in \h$, we say that $x$ is $h$-contributed if there exists $A \in \R$ such that $x$ is $(A,h)$-contributed (see Figure~\ref{Fig:Contributed}(a)).
\end{definition}

\begin{observation}\label{Obs:Aux}
\begin{enumerate}
\item For any given $A,h$, and for any $0 \leq s \leq k$, there exists at most one point $x$ with $\mathrm{depth}(x)=s$ that is $(A,h)$-contributed.

\item It may be that $x$ that is $(A,h)$-contributed but $A \cap h = \emptyset$ (see Figure~\ref{Fig:Contributed}(b)).
\end{enumerate}
\end{observation}

\begin{definition}
An $(A,h)$-contributed point $x$ is called an \emph{inner contribution} of $A$ if there exist points $y,z$ and lines $h \neq h',h'' \in \h$, such that:
\begin{itemize}
\item $y$ is $(A,h')$-contributed and $z$ is $(A,h'')$-contributed, and

\item $x$ lies strictly between $y$ and $z$. (Note that all of $x,y,z$ belong to the upper boundary of $A$. This induces a natural ordering between them.)
\end{itemize}
If there are no such points, $x$ is called an \emph{extremal contribution} of $A$ (see Figure~\ref{Fig:Contributed}(c)).
\end{definition}

\begin{figure}[tb]
\begin{center}
\scalebox{0.6}{
\includegraphics{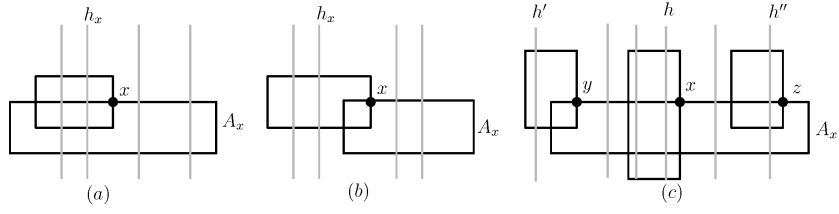}
} \caption{An auxiliary figure for Section~\ref{sec:sub:contributed}. In~(a) and~(b), the point $x$ is $(A_x,h_x)$-contributed. In~(c), $x \in X_0$ is an inner contribution of $A$.}
\label{Fig:Contributed}
\end{center}
\end{figure}

The following observation is crucial in the sequel.
\begin{observation}\label{Obs:Inner-contribution}
Let $x \in X_{\leq k}$ be an $(A,h_{i})$-contributed intersection point. If $x$ is an inner contribution of $A$, then $A$ intersects both $h_i$ and $h_{i+1}$.
\end{observation}

\begin{proof}
Denote the vertical lines that contain the left and right boundaries of $A$ by $l_A$ and $r_A$, respectively. Note that if for some $m$ there exists an $(A,h_m)$-contributed point $\bar{x}$, then the line $h_{m+1}$ must lie to the right of $l_A$ (as otherwise, $B_{\bar{x}}$ must intersect $h_{m+s}$ for some $s \geq 1$, contradicting the assumption that $\bar{x}$ is contributed by $h_m$). On the other hand, $h_m$ must lie to the left of $r_A$, since it intersects $B_{\bar{x}}$ and the right boundary of $B_{\bar{x}}$ is to the left of $r_A$ (as the intersection point $\bar{x}$ is of type $L$, see Figure~\ref{fig:types-of-intersection}(a)).

In our case, as $x$ is an inner contribution of $A$, there exist some $s_1,s_2 \geq 1$ and points $y,z$ such that $y$ is $(A,h_{i-s_1})$-contributed and $z$ is $(A,h_{i+s_2})$-contributed. By the previous paragraph, the former implies that $h_{i-s_1+1}$ lies to the right of $l_A$ while $h_{i+s_2}$ lies to the left of $r_A$. As $s_1,s_2 \geq 1$, this implies that both $h_i$ and $h_{i+1}$ lie to the right of $l_A$ and to the left of $r_A$, and thus, both intersect $A$, as asserted.
\end{proof}

\subsection{Upper bound on `inner contributions' to the $(\leq k)$-level complexity of $\R$}
\label{sec:sub:floors-contribution}

%It is clear from the definitions that for any $R \in \R$, at most two intersection points are extremal contributions to $R$, and thus overall, at most $2|\R|$ elements of $X_{\leq k}$ are extremal contributions. Therefore, it will be sufficient for us to obtain an upper bound on the number of elements of $X_{\leq k}$ that are inner contributions.
In this subsection we obtain an upper bound on the number of elements of $X_{\leq k}$ that are inner contributions, by considering pairs of the form (Floor $\F_i$, vertical line $h_j$) separately, and for each such pair, upper bounding the number of $(A,h_j)$-contributed points for $A \in \F_i$ that are inner contributions.
%We now consider all extremal points contributed by rectangles of the same floor, and use the observations presented above to obtain an upper bound on their contribution to the $(\leq k)$-level complexity of $\R$.
\begin{proposition}\label{Prop:Inner-contributions}
For $k \geq 0$ and $1 \leq i,j \leq p-1$, let
\[
S_k^{i,j} = \{x \in X_{\leq k}: \exists A \in \F_i, x \mbox{ is } (A,h_j)\mbox{-contributed and } x \mbox{ is an inner contribution of } A\}.
\]
(Informally, $S_k^{i,j}$ is the set of all contributions to the level $\leq k$ complexity, that are contributed by $h_j$ on the $i$'th floor in an `inner' way). Then for all $i,j$,
\begin{equation}\label{Eq:Floors-contribution1}
|S_k^{i,j}| \leq \frac{(k+1)(k+2)}{2}.
\end{equation}
\end{proposition}

\begin{figure}[tb]
\begin{center}
\scalebox{0.6}{
\includegraphics{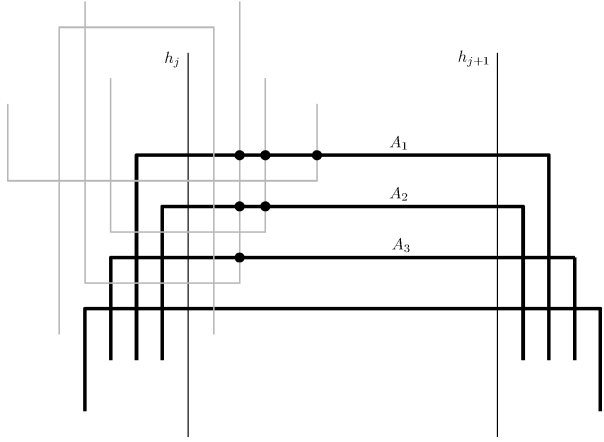}
} \caption{An illustration to the proof of Proposition~\ref{Prop:Inner-contributions}.}
\label{fig:Floors-contribution}
\end{center}
\end{figure}

\begin{proof}
Fix $1 \leq j \leq p-1$. Define, for any $1 \leq i \leq p-1$,
\[
\A_i = \{A \in \F_i: \exists (A,h_j)\mbox{-contributed } x \in X_{\leq k} \mbox{ that is an inner contribution of } A\}.
\]
(Informally, $\A_i$ is the set of all rectangles on the $i$'th floor, whose upper edge contains an inner contribution to the level $\leq k$ complexity, contributed by $h_j$.)
Denote $|\A_i|=m_i$, and let the elements of $\A_i=\{A_1,A_2,\ldots,A_{m_i}\}$ be ordered in descending order of the height of the upper boundary, as demonstrated in Figure~\ref{fig:Floors-contribution}. (So, $A_1$ is the rectangle whose upper boundary is the highest, $A_2$'s upper boundary is the second highest, etc.. Note that equality cannot occur here as by Observation~\ref{Obs:Inner-contribution}, any $A \in \A_i$ intersects both $h_j$ and $h_{j+1}$, and so, if two of these rectangles had upper boundaries of the same height, they would share part of the boundary, contradicting the assumption that the elements of $\R$ are in general position.)

\medskip \noindent For each $1 \leq l \leq m_i$, denote
\[
Q_l=\{x \in X_{\leq k}: x \mbox{ is } (A_l,h_j) \mbox{-contributed}\}.
\]
Note that we have
\begin{equation}\label{Eq:Floors-contribution2}
|S_k^{i,j}| \leq |\{x \in X_{\leq k}: \exists A \in \A_i \mbox{ such that } x \mbox{ is } (A,h_j)\mbox{-contributed}\}| = \sum_{l=1}^{m_i} |Q_l|.
\end{equation}
Let $x \in Q_l$. We claim that for each $1 \leq r \leq l-1$, $x$ is an interior point of $A_r$. To see this, we need several simple observations.
\begin{enumerate}
\item Any $(A_l,h_j)$-contributed $x$ lies between the lines $h_j$ (inclusive) and $h_{j+1}$ (non-inclusive). Indeed, as $x$ lies on the right boundary of $B_x$ and $h_j$ intersects $B_x$, $x$ must lie either on $h_j$ or to the right of $h_j$. On the other hand, if $x$ lies on $h_{j+1}$ or on the right of $h_{j+1}$, then $B_x$ must intersect $h_{j+s}$ for some $s \geq 1$, a contradiction.

\item Any such $x$ lies above or on the line $\ell_{i}$, since it belongs to the upper boundary of $A_l \in \F_i$.

\item Each of the rectangles $A_1,\ldots,A_{m_i}$ intersects $\ell_i$ by Observation~\ref{Obs:Floors}, and intersects both $h_j$ and $h_{j+1}$ by Observation~\ref{Obs:Inner-contribution}.
\end{enumerate}
By the simple observations, for each $1 \leq r \leq l-1$, the rectangle $A_r$ intersects $\ell_i$, $h_j$ and $h_{j+1}$, and its upper boundary lies above $x$ (since $x$ lies on the upper boundary of $A_l$). As $x$ lies between the lines $h_j$ and $h_{j+1}$ and above $\ell_i$, it follows that $x$ is an interior point of $A_r$.

\medskip \noindent Now, arrange the elements of $Q_l = \{x_1,x_2,\ldots\}$ in descending order of the $x$ coordinate (i.e., $x_1$ is the rightmost one, $x_2$ is the second-to-right, etc. Such an ordering is possible, since all elements of $Q_l$ belong to the upper boundary of $A_l$).  For each $x \in Q_l$, $B_x$ intersects $h_j$ (since $x$ is $(A_l,h_j)$-contributed). Thus, $x_{m}$ is included in the interior of $B_{x_{m'}}$ for any $m>m'$. In addition, for any $1 \leq t \leq l-1$, all elements of $Q_l$ are interior points of $A_t$. Therefore, for any $m \geq 1$, we have $\mathrm{depth}(x_m) \geq (m-1)+(l-1)=l+m-2$. As all points in $Q_l$ are of depth $\leq k$, this implies $|Q_l| \leq k-l+2$ for any $1 \leq l \leq m_i$. Summing over all values of $l$ and using~\eqref{Eq:Floors-contribution2}, we obtain
\begin{equation}
|S_k^{i,j}| \leq \sum_{l=1}^{m_i} |Q_l| \leq (k+1)+k+\ldots+1 = \frac{(k+1)(k+2)}{2},
\end{equation}
as asserted.
\end{proof}

\subsection{Finalizing the proof of Theorem~\ref{Thm:Main}}
\label{sec:sub:final}

Now we are ready to prove Theorem~\ref{Thm:Main}. Actually, we prove the following exact version of the theorem:
\begin{theorem}\label{Thm:Exact}
Let $\R$ be a family of $n$ axis-parallel rectangles with $\nu(\R)=p-1$. For any $k \geq 0$, the $(\leq k)$-level complexity of $\R$ is at most $8(k+1)n+2(p-1)(p-3)(k+1)(k+2)$. In particular, the union complexity of $\R$ is at most $8n+4(p-1)(p-3)$.
\end{theorem}

\begin{proof}
By symmetry considerations, the $(\leq k)$-level complexity of $\R$ is at most $4|X_{\leq k}|$, so it is sufficient to prove
\begin{equation}\label{Eq:Final1}
|X_{\leq k}| \leq 2(k+1)n+(p-1)(p-3)(k+1)(k+2)/2.
\end{equation}
We prove~\eqref{Eq:Final1} by upper bounding the inner contributions and the extremal contributions separately.

\medskip \noindent \textbf{Inner contributions.} By Proposition~\ref{Prop:Inner-contributions}, for each $i,j$, the number of inner contributions that correspond to $\F_i$ and $h_j$ is at most $(k+1)(k+2)/2$. For $j \in \{1,p-1\}$, any $h_j$-contributed $x$ is an extremal contribution. Hence, the number of inner contributions that correspond to $\F_i$ is at most $(p-3)(k+1)(k+2)/2$, and so, the total number of inner contributions is at most $(p-1)(p-3)(k+1)(k+2)/2$.

\medskip \noindent \textbf{Extremal contributions.} Let $A \in \R$. By the definition of inner and extremal contributions, all $A$-contributed points that are extremal contributions belong to one of two vertical lines. By Observation~\ref{Obs:Aux}, for any single pair $(A,h)$, $X_{\leq k}$ contains at most $k+1$ $(A,h)$-contributed points. Therefore, the number of $A$-contributed points that are extremal contributions is at most $2(k+1)$. It follows that the total number of extremal contributions is at most $2(k+1)n$. This completes the proof.
\end{proof}

\begin{remark}
If one is interested in the $k$-level complexity of $\R$ (instead of the $(\leq k)$-level complexity we treat), the same proof method can be used to show that it is at most $O(n+kp^2)$, and that this is tight for the example presented in Figure~\ref{fig:tightness} below.
\end{remark}

\begin{remark}
We note that an alternative way to prove Theorem~\ref{Thm:Main} is to first obtain an upper bound on the union complexity of $\R$ and then deduce an upper bound on the $(\leq k)$-level complexity by the classical technique of Clarkson and Shor~\cite{CS89}. We preferred to treat the $(\leq k)$-level complexity directly, as this allows obtaining the slightly better exact result of Theorem~\ref{Thm:Exact} with almost the same effort.
\end{remark}

\section{Tightness of Theorem~\ref{Thm:Main}}
\label{sec:tightness}

In this section we present a family $\R$ of $n$ axis-parallel rectangles with $\nu(\R)=p-1$ whose $(\leq k)$-level complexity is $\Theta(nk+k^2p^2)$, thus showing that Theorem~\ref{Thm:Main} is tight (up to a constant factor).

The family $\R$, presented in Figure~\ref{fig:tightness}, is a disjoint union of two subfamilies of $n/2$ rectangles each.

\begin{figure}[tb]
\begin{center}
\scalebox{0.6}{
\includegraphics{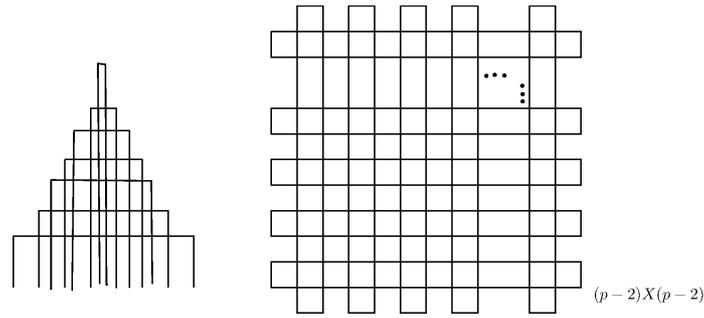}
} \caption{A family of axis-parallel rectangles that demonstrates the tightness of Theorem~\ref{Thm:Main}.}
\label{fig:tightness}
\end{center}
\end{figure}

The subfamily drawn in the left of the figure consists of a sequence of pairwise-intersecting rectangles in which each rectangle is taller and thinner than its successor. This subfamily contributes $O(kn)$ points to the $(\leq k)$-level complexity of $\R$.

The subfamily drawn in the right of the figure is based on an $(p-2)$-by-$(p-2)$ grid of long thin rectangles. We replace each rectangle in the basic grid with $\frac{n}{4(p-2)}$ nested copies to obtain a family of $n/2$ rectangles (for simplicity, we assume $4(p-2)|n$; note that only the basic grid is depicted in the figure). This subfamily contributes $\Theta(k^2p^2)$ points to the $(\leq k)$-level complexity of $\R$.

Hence, the $(\leq k)$-level complexity of $\R$ is $\Theta(nk+k^2p^2)$, as asserted.

%\begin{remark}
%The subfamily $\R_2$ drawn in the right of Figure~\ref{fig:tightness} emphasizes that the assertion of Theorem~\ref{Thm:Main} is not hereditary. Indeed, for $n=O(p)$, $\R_2$ satisfies $\nu(\R_2)=p$, and still, the union complexity of $\R_2$ is quadratic in its size.
%\end{remark}

\end{document}